\documentclass[12pt]{article}

\usepackage{amsmath,amssymb,amsfonts,amsthm}
\usepackage{tikz,color,pgf,graphicx,url}

\newtheorem{theorem}{Theorem}[section]
\newtheorem{corollary}[theorem]{Corollary}

\newtheorem{proposition}[theorem]{Proposition}

\newtheorem{question}[theorem]{Question}

\def\cp{\,\square\,}

\newcommand{\gi}{\gamma_{\rm i}} 
\newcommand{\ec}{\rho} 
\newcommand{\ggr}{\gamma_{\rm gr}} 
\newcommand{\IR}{{\rm IR}} 
\renewcommand{\gg}{\gamma_{\rm g}}

\begin{document}

\title{Indicated domination game}

\author{Bo\v stjan Bre\v sar$^{a,b}$\thanks{Email: \texttt{bostjan.bresar@um.si}}
\and Csilla Bujt\'as$^{b,c,d}$\thanks{Email: \texttt{bujtasc@gmail.com}}
\and Vesna Ir\v si\v c$^{b,c}$\thanks{Email: \texttt{vesna.irsic@fmf.uni-lj.si}}
\and Douglas F.\ Rall$^{e}$\thanks{Email: \texttt{doug.rall@furman.edu}}
\and Zsolt Tuza$^{d,f}$\thanks{Email: \texttt{tuza.zsolt@mik.uni-pannon.hu}}}
\maketitle

\begin{center}
$^a$ Faculty of Natural Sciences and Mathematics, University of Maribor, Slovenia\\
\medskip

$^b$ Institute of Mathematics, Physics and Mechanics, Ljubljana, Slovenia\\
\medskip

$^c$ Faculty of Mathematics and Physics, University of Ljubljana, Slovenia\\
\medskip

$^d$ University of Pannonia, Veszpr\'em, Hungary\\
\medskip

$^e$ Furman University, USA\\
\medskip

$^f$ Alfr\'ed R\'enyi Institute of Mathematics, Budapest, Hungary\\
\medskip

\end{center}

\begin{abstract}
Motivated by the success of domination games and by a variation of the coloring game called the indicated coloring game, we introduce a version of domination games called the indicated domination game. It is played on an arbitrary graph $G$ by two players, Dominator and Staller, where Dominator wants to finish the game in as few rounds as possible while Staller wants just the opposite. In each round, Dominator indicates a vertex $u$ of $G$ that has not been dominated by previous selections of Staller, which, by the rules of the game, forces Staller to select a vertex in the closed neighborhood of $u$. The game is finished when all vertices of $G$ become dominated by the vertices selected by Staller. Assuming that both players are playing optimally according to their goals, the number of selected vertices during the game is the indicated domination number, $\gamma_{\rm i}(G)$, of $G$. 

We prove several bounds on the indicated domination number expressed in terms of other graph invariants. In particular, we find a place of the new graph invariant in the well-known domination chain, by showing that $\gamma_{\rm i}(G)\ge \Gamma(G)$ for all graphs $G$, and by showing that the indicated domination number is incomparable with the game domination number and also with the upper irredundance number. In connection with the trivial upper bound $\gamma_{\rm i}(G)\le n(G)-\delta(G)$, we characterize the class of graphs $G$ attaining the bound provided that $n(G)\ge 2\delta(G)+2$. We prove that in trees, split graphs and grids the indicated domination number equals the independence number. We also find a formula for the indicated domination number of powers of paths, from which we derive that there exist graphs in which the indicated domination number is arbitrarily larger than the upper irredundance number. We provide some partial results supporting the statement that $\gamma_{\rm i}(G)=n(G)/2$ if $G$ is a cubic bipartite graph, and leave this as an open question. 
\end{abstract}


\noindent
{\bf AMS Subj.\ Class.\ (2020)}: 05C57

\section{Introduction}
\label{sec:intro}

The coloring game was introduced independently in \cite{ga-81} and \cite{bo-1991}. Unlike combinatorial games in which a winner is decided, the result of the coloring game gives a graph invariant, which is based on the assumption that both players are playing optimally according to their goals. A number of variants of the original game have been introduced, see e.g.\ \cite{bagr-07, borowiecki+2007listcol, mpw-2018, tz-2015}.
The following version was proposed by Grytczuk and defined by Grzesik~\cite{gr-2012}.

The \emph{indicated coloring game} is played on a simple graph $G$ by two players, and a fixed set $C$ of colors. In each round of the game Ann indicates an uncolored vertex, and Ben colors it using a color from $C$, obeying just the proper coloring rule. The goal of Ann is to achieve a proper coloring of the whole graph, while Ben is trying to prevent this. The minimum cardinality of the set of colors $C$ for which Ann has a winning strategy is called the \emph{indicated chromatic number}, $\chi_i(G)$, of a graph $G$.

The \emph{domination game}, as introduced in~\cite{bresar-2010}, is played on a graph $G$ by two players: Dominator and Staller. They alternate taking moves in which they select a vertex of $G$. A move is legal if the selected vertex dominates at least one vertex which is not already dominated by previously played vertices. The game ends when there are no legal moves, so when the set of played vertices is a dominating set of $G$. The goal of Dominator is to finish the game with the minimum number of moves, while the aim of Staller is to maximize the number of moves. If both players play optimally, then the number of moves played on $G$ is an invariant (see~\cite{book-2021, fractional-2019}).
Therefore, as defined in~\cite{bresar-2010}, the \emph{game domination number} $\gg(G)$ is the number of moves on $G$ if Dominator starts the game. Many variants of the domination game have been introduced; see a recent monograph~\cite{book-2021} and for example papers~\cite{borowiecki+2019connected,bujtas-2022a, bujtas-2022b, bujtas-2021,fractional-2019, duchene-2020, gledel-2020, henning-2015, xu-2018a, xu-2018b}. Having in mind the indicated coloring game, we propose the following variant of the domination game.

The \emph{indicated domination game} is played on a graph $G$ by two players, Dominator and Staller, who take turns making a move. In each of his moves, Dominator indicates a vertex $v$ of the graph that has not been dominated in the previous moves, and Staller chooses (or selects) any vertex from the closed neighborhood of $v$, and adds it to
a set $D$ that is being built during the game. The game ends when there is no undominated vertex left, that is, when $D$ is a dominating set. The goal of Dominator is to minimize the size of $D$, while Staller wants just the opposite. Provided that both players are playing optimally with respect to their goals, the size of the resulting set $D$ is the \emph{indicated domination number} of $G$, and is denoted by $\gi(G)$.

In the following section, we establish the notation and present basic definitions, while in Section~\ref{sec:prelim} we give some preliminary results. In particular, we prove that $\Gamma(G)\le \gi(G)\le \ggr(G)$, where $\Gamma(G)$ is the upper domination number and $\ggr(G)$ is the Grundy domination number of a graph $G$. In Section~\ref{sec:upper}, we prove that for a graph $G$ with minimum degree $\delta$ and order $n$, where $n \geq 2 \delta + 2$, we have $\gi(G) = n - \delta$ if and only if $G$ contains a spanning subgraph $K_{\delta, n - \delta}$ with an additional property that the part of the bipartition of size $n - \delta$ is an independent set in $G$. In Section~\ref{sec:alpha}, we prove that in several families of graphs (namely trees, split graphs, grids, and connected bipartite cubic graphs with at most 12 vertices) the indicated domination number equals the independence number. On the other hand, the indicated domination number can be arbitrarily larger than the upper irredundance number (and thus also the independence number), which is established in Section~\ref{sec:powers}. This is derived from the formula for the indicated domination number of the $k$-th power of the path $P_n$, which is roughly $\gi(P_n^k) = \Theta\Big(\frac{\log k}{k} \, n\Big)$ as $n\to\infty$.	In Section~\ref{sec:conclude}, we propose several open questions.

\section{Notation}
\label{sec:notation}

Let $G$ be a graph. We denote the number of vertices of $G$ by $n(G)$. If $S \subseteq V(G)$, then the subgraph induced by $S$ is denoted by $G[S]$. For a vertex $v \in V(G)$, the \emph{(open) neighborhood} $N(v)$ is the set of neighbors of $v$, and the \emph{closed neighborhood} is $N[v] = N(v) \cup \{v\}$. If $S \subseteq V(G)$, then $N[S] = \bigcup_{v\in S} N[v]$.

For a vertex $x \in S$, every vertex in $N[S] - N[S-\{x\}]$ is called a \emph{private neighbor of $x$ with respect to $S$}.\footnote{Observe that every vertex isolated in $G[S]$ is viewed as a private neighbor of itself by definition.} A set $S \subseteq V(G)$ is an \emph{irredundant set} if every vertex in $S$ has a private neighbor with respect to $S$. The smallest and largest cardinalities of a maximal irredundant set of $G$ are denoted by ${\rm ir}(G)$ and $\IR(G)$, respectively.

A set $S \subseteq V(G)$ is an \emph{independent set} in a graph $G$ if the vertices in $S$ are pairwise nonadjacent. The maximum size of an independent set of $G$ is denoted by $\alpha(G)$.  An \emph{edge cover} of $G$ is a set $F \subseteq E(G)$ such that every vertex of $G$ is incident to some edge in $F$. We denote the minimum size of an edge cover of $G$ by $\ec(G)$. Note that notation $\beta'(G)$ is also used in the literature. Recall that combining K\"onig's and Gallai's   Theorems gives $\alpha(G) = \ec(G)$ for a bipartite graph $G$ without isolated vertices.

A vertex $v \in V(G)$ \emph{dominates} itself and its neighbors. A subset of vertices $D \subseteq V(G)$ is a \emph{dominating set} of $G$ if it dominates all vertices of $G$, i.e.\ $N[D] = V(G)$. This means that every vertex from $V(G) - D$ has a neighbor in $D$. The minimum cardinality of a dominating set of $G$ is the \emph{domination number}, $\gamma(G)$, of $G$.

A dominating set $D$ in $G$ is a \emph{minimal dominating set} if no proper subset of $D$ is a dominating set. That is, $D$ is a minimal dominating set if and only if every $x \in D$ has a private neighbor with respect to $D$. The maximum cardinality of a minimal dominating set is the \emph{upper domination number}, $\Gamma(G)$, of $G$. We recall the following results.

\begin{theorem}[{\cite[Theorem 5]{co-fa-pa-th-1981}}] \label{thm:topthreebipartite}
	If $G$ is a bipartite graph, then $\alpha(G)=\Gamma(G)= \IR(G)$.
\end{theorem}

\begin{theorem}[{\cite[Theorem 9]{jacobsen-1990}}]
	\label{thm:topthreechordal}
	If $G$ is a chordal graph, then $\alpha(G)=\Gamma(G)= \IR(G)$.
\end{theorem}

Given a graph $G$, a sequence $S=(v_1,\ldots,v_k)$ of vertices of $G$ is a {\em dominating sequence} if for each $i$
\begin{equation}
	\label{eq:grundy}
	N[v_i] - \cup_{j=1}^{i-1}N[v_j]\not=\emptyset
\end{equation}
and the set of vertices from $S$ dominates $G$.
We call the length $k$ of the longest such sequence $S$ the {\em Grundy domination number}, $\ggr(G)$, of $G$. Clearly, $\gamma(G) \leq \ggr(G)$.

\section{Preliminary results}
\label{sec:prelim}

A well known domination chain~\cite{haynes-1998} was extended with the Grundy domination  number in~\cite{bresar-2014}:
\[ {\rm ir}(G) \le \gamma(G) \le i(G) \le \alpha(G) \le \Gamma(G) \le {\rm IR}(G) \le \ggr(G)\,.\]
The following proposition explains where indicated domination number fits in this chain.

\begin{proposition}
	\label{prop:diagram}
	If $G$ is a graph, then
	$$\Gamma(G) \le \gi(G) \le \ggr(G)\,.$$
	In addition, $\gi$ is incomparable with both $\gg$ and $\IR$.
\end{proposition}

\begin{proof}
	Note that the definition of the indicated domination game implies that the sequence of vertices selected by Staller is a dominating sequence, so $\gi(G) \leq \ggr(G)$.
	
	Let $D$ be a minimal dominating set of cardinality $\Gamma(G)$. We provide a strategy for Staller that will prove that $\Gamma(G) \le \gi(G)$.  The basic ingredient of this strategy is for her to always select a vertex from $D$. She can do so regardless of which vertex Dominator indicates. Unless all vertices from $D$ have been selected, there exists an undominated vertex (a private neighbor of an unselected vertex from $D$). This shows that $\Gamma(G) \le \gi(G)$.
	
	In the rest of the proof we present infinite families of graphs showing that  $\gi$ is incomparable with both $\gg$ and $\IR$. Stars $K_{1,n}$ show that $\gi$ can be arbitrarily larger than $\gg$, since $\gi(K_{1,n}) = n$ and $\gg(K_{1,n}) = 1$. Second powers of paths (see Corollary~\ref{cor:gi-IR-plus}) show that $\gi$ can be arbitrarily larger than $\IR$.
	
	To see that $\IR$ can be arbitrarily larger than $\gi$, we consider the following family of graphs. Let $H_n$ be the graph obtained from two disjoint copies of $K_n$ with vertices $u_1, \ldots, u_n$ and $v_1, \ldots, v_n$ by adding edges $u_i v_i$ for all $i \in \{2, \ldots, n\}$. Notice that $H_n$ is isomorphic to the graph $K_n \cp K_2$ with one edge removed. It is easy to see that $\IR(H_n) = n-1$ (it is attained for example by the maximal irredundant set $\{u_2, \ldots, u_n\}$). On the other hand, we have $\gi(H_n) \geq \gamma(H_n) = 2$, and if Dominator first indicates $u_1$ and then $v_1$, we get $\gi(H_n) = 2$. Note that this also shows that in general $\gi(G)$ and $\gi(G-e)$ can be arbitrarily far apart since $\gi(K_n \cp K_2) \ge \Gamma(K_n \cp K_2)=n$.

	Last, we present a family of graphs for which $\gg$ is arbitrarily larger than $\gi$. Let $D_1$ be a graph obtained from a cycle $C_9$ with vertices $x_1, \ldots, x_9$ and naturally defined edges between them by adding edges $x_1 x_3$, $x_4 x_6$ and $x_7 x_9$. Since Dominator can in turn indicate vertices $x_2$, $x_5$ and $x_8$, he has a strategy that ensures $\gi(D_1) \leq 3$. Since $\gi(D_1) \geq \gamma(D_1) = 3$, we have $\gi(D_1) = 3$. On the other hand, it is easy to see that $\gg(D_1) = 4$. Let $D_n$ be a disjoint union of $n$ copies of the graph $D_1$. Since $\gi(D_n) \geq \gamma(D_n) = 3n$ and Dominator can in turn indicate vertices of degree two, we have $\gi(D_n) = 3n$. But since Staller's strategy in the domination game on $D_n$ can be to reply in the same copy of $D_1$ as Dominator played, we get $\gg(D_n) \geq 4n$ (because under that strategy, in each copy the game ends after an even number, four, of moves). Thus $\gg(D_n) - \gi(D_n) \geq n$. Alternatively, we can consider a connected graph $E_n$ obtained from $C_{3n}$ by adding edges $x_1 x_3, x_4 x_6, \ldots, x_{3n-2} x_{3n}$. In this case, $\gi(E_n) = n$, while $\gg(E_n) \geq n + \frac{n}{10}$ (the last follows by an argument similar as in~\cite[Theorem 4.1]{kosmrlj-2014}).
\end{proof}

The diagram in Figure~\ref{fig:hasse} shows the relations between the parameters from Proposition~\ref{prop:diagram}.

\begin{figure}[ht!]
	\begin{center}
		\begin{tikzpicture}
			\node (0) at (0,0) {${\rm ir}$};
			\node (1) at (0,1) {$\gamma$};
			\node (2) at (0,2) {$i$};
			\node (3) at (0,3) {$\alpha$};
			\node (4) at (0,4) {$\Gamma$};
			\node (5) at (-1.5,5) {$\mathbf{\gi}$}; 
			\node (6) at (0,5) {$\IR$};
			\node (7) at (1.5,5) {$\gg$};
			\node (8) at (0,6) {$\ggr$};
			
			\draw (0) -- (1) -- (2) -- (3) -- (4) -- (5) -- (8);
			\draw (4) -- (6) -- (8);
			\draw (1) -- (7) -- (8);
			
		\end{tikzpicture}
	\end{center}
	\caption{Relations between the domination and independence invariants.}
	\label{fig:hasse}
\end{figure}
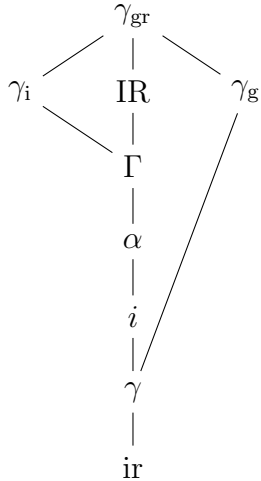

There are several classes of graphs with a transparent structure where the parameters discussed in Proposition~\ref{prop:diagram} coincide. For example, the graphs $G$ below all have the property $\Gamma(G)=\ggr(G)$ as listed in~\cite{bacso-2023}: hypercubes; complete multipartite graphs, $G=K_{n_1,\ldots,n_k}$, such that $n_1 \ge \cdots \ge n_k$, $k \ge 2$, and $n_{k-1}\ge 2$; prisms over complete graphs, $G= K_n \cp K_2$, for $n \ge 2$; large families of Kneser graphs; the class of (twin-free, connected) cographs. It is also proved in~\cite{bacso-2023} that the join operation preserves that property. By Proposition~\ref{prop:diagram} all the graphs from the above families have the indicated domination number equal to their upper domination number. By Theorem~\ref{thm:topthreebipartite} this in turn implies that the bipartite graphs among the above families have their indicated domination number equal to their independence number.

However, the difference between the indicated domination number and the independence number can be arbitrarily large. For example, $\alpha(K_n \cp K_2) = 2$ and $\gi(K_n \cp K_2) = n$. Similarly, the difference between the Grundy domination number and the indicated domination number can be arbitrarily large. For example, $\ggr(P_n) = n-1$ and $\gi(P_n) = \left \lceil \frac{n}{2} \right \rceil$ as will be proven in Corollary~\ref{cor:paths}.

\section{Upper bound on $\gi$ and extremal graphs}
\label{sec:upper}

Using a bound from Proposition~\ref{prop:diagram} combined with~\cite[Proposition 2.1]{bresar-2014} we obtain:
\begin{equation}
\label{eq:bounds}
\gi(G) \leq \ggr(G) \leq n(G) - \delta(G).
\end{equation}

In the following theorem we characterize the graphs $G$ that attain the bound $\gi(G) = n(G) - \delta(G)$
under the assumption that $n(G) \ge 2\delta(G)+2$.

\begin{theorem}
	\label{thm:extremal}
	A graph $G$ with minimum degree $\delta$ and order $n$, where $n \geq 2 \delta + 2$, has $\gi(G) = n - \delta$ if and only if it contains a spanning subgraph $K_{\delta, n - \delta}$ with an additional property that the part of the bipartition of size $n - \delta$ is an independent set in $G$.
\end{theorem}

\begin{proof}
Let $G$ be a graph, which has a spanning subgraph $K_{\delta, n - \delta}$ and the only edges of $G$ that are not in $K_{\delta, n - \delta}$, if any, are between two vertices in the part of size $\delta$. Let the indicated domination game be played on $G$. Regardless of which vertex is indicated by Dominator in his first move, Staller selects a vertex from the part of size $n-\delta$. In this way, all vertices in the part of size $\delta$ become dominated after the first selection of Staller. Therefore, Dominator will in his subsequent moves need to indicate only vertices in the part of size $n-\delta$, and Staller's strategy is simply to select the vertex that Dominator indicated in his move. In this way, Staller ensures that at least $n-\delta$ vertices will be chosen. Thus, $\gi(G)\ge n - \delta$, and combining this with~\eqref{eq:bounds} we get  $\gi(G) =n - \delta$.

For the converse, let $G$ be a graph on $n$ vertices, where $n \geq 2 \delta + 2$, such that $\gi(G)=n - \delta$. Let the indicated domination game be played on $G$. If after the first selection of Staller more than $\delta+1$ vertices were dominated, then at most $n-\delta-1$ vertices will be chosen by Staller when the indicated domination game ends on $G$,  a contradiction. Therefore, Staller needs to select a vertex of degree $\delta$ in her first move. For the same reason, since $\gi(G)=n-\delta$, in all of her subsequent moves Staller needs to select a vertex that dominates only one vertex that was not dominated earlier. Let $x$ be the vertex selected in the first move of Staller, and let $N_G(x)=\{u_1,\ldots, u_\delta\}$. Let $R=V(G)-N_G[x]$. If there existed a vertex $v\in R$ such that $v$ had no neighbors in $N_G(x)$, then Dominator's next move would be to indicate $v$, and after Staller selected a vertex that dominates $v$, at least two new vertices would become dominated, a contradiction. Therefore, every vertex in $R$ has at least one neighbor in $N_G(x)$.

Now, suppose that there exists a vertex in $R$, say $w_0$, that has a neighbor in $R$, and let $w_1,\ldots,w_k$ be the neighbors of $w_0$ from $R$. For every vertex $w_i$, where $i\in\{0,\ldots, k\}$, there exists a vertex $u_{j_i}\in N_G(x)$ such that $w_i$ is the only  neighbor of $u_{j_i}$ in $R$. Indeed, if for some $w_i$ such  a vertex $u_{j_i}$ would not exist, then Dominator could indicate $w_i$ in his second move, and regardless of which vertex Staller would select in her second move to dominate $w_i$ at least two new vertices would become dominated in that move, which is a contradiction. Since $\deg(w_0)\ge \delta$, we infer that $w_0$ has (at least) $\delta-k$ neighbors in $N_G(x)$. Clearly, all of the neighbors of $w_0$ in $N_G(x)$ are distinct from vertices $u_{j_i}$, where $i\in [k]$, which implies that $w_0$ has exactly $\delta-k$ neighbors in $N_G(x)$. In addition, $k\le \delta -1$, since all vertices $u_{j_i}$, where $i\in \{0,1,\ldots,k\}$, are pairwise distinct.  Therefore, since $n>2\delta+1 \geq |N_G[x]|+k+1$, we infer that there exists a vertex $z\in R-\{w_0,\ldots,w_k\}$. If $z$ has a neighbor in $R$, then, by the same reason as earlier, there exists a vertex $u'$ in $N_G(x)$ such that $z$ is the only neighbor of $u'$ from $R$. However, every vertex in $N_G(x)$ has a neighbor in $\{w_0,\ldots,w_k\}$, therefore $z$ cannot be adjacent to any vertex in $R$. Since $\deg(z)\ge \delta$, we infer $N_G(z)=N_G(x)$, which is again a contradiction, since each $u_{j_i} \in N_G(x)$ has only one neighbor in $R$.
The assumption that there exists an edge between two vertices in $R$ led us to a contradiction, therefore $R\cup\{x\}$ is independent, and $K_{\delta,n-\delta}$ is a spanning subgraph of $G$ with the part of the bipartition of size $n-\delta$ being independent.
\end{proof}

Note that a graph $G$ with minimum degree $\delta$ and order $n$, where $n \geq 2 \delta + 2$ and $\gi(G) = n - \delta$, can equivalently be described as being the join of an independent set of size $n-\delta$ and any graph of order $\delta$.

\section{Graphs with $\gi=\alpha$}
\label{sec:alpha}
In this section, we present several classes of graphs whose indicated domination number equals their independence number. Naturally, due to Proposition~\ref{prop:diagram}, the independence number of any of these graphs is equal to the upper domination number.

Recall that a graph is {\it split} if its vertex set can be partitioned into a clique and an independent set.
A {\it split partition} of a split graph $G$ is a pair $(K,I)$ such that $K$ is a clique, $I$ is an independent set,
$K\cup I = V(G)$ and $K\cap I = \emptyset$.
It was proven in~\cite{bresar-2014} that if $G$ is a split graph with split partition $(K,I)$, then
$$\gamma_{gr}(G) =\left\{
	\begin{array}{ll}
		\alpha(G), & \hbox{if every two vertices in $K$ have a common neighbor in $I$;} \\
		\alpha(G)+1, & \hbox{otherwise.}
	\end{array}
	\right.$$
We next show that the second line in the above equation never appears if $\ggr$ is replaced with $\gi$.

\begin{proposition}
	\label{prop:split}
	If $G$ is a split graph, then $\gi(G) = \alpha(G)$.
\end{proposition}

\begin{proof}
	Let $(K, I)$ be the split partition of $G$. We can assume that $I$ is a maximum independent set, and thus every vertex in $K$ has a neighbor in $I$. Dominator's strategy
in each of his moves is to indicate a vertex $x$ from $I$ that has not been dominated by the vertices previously chosen by Staller.  After all vertices from $I$ have been dominated by Staller's selected vertices, the vertices in $K$ have also been dominated since $K$ is a clique and each vertex from $K$ has a neighbor in $I$.  Hence at most $|I| = \alpha(G)$ vertices were chosen by Staller, and we infer that $\gi(G) \leq \alpha(G)$.  The reverse inequality follows from Proposition~\ref{prop:diagram}.
\end{proof}

To see that trees enjoy the equality considered in this section is a bit less straightforward.

\begin{theorem}
	\label{thm:trees}
	If $T$ is a tree, then $\gi(T) = \alpha(T)$.
\end{theorem}
\begin{proof}
Let $F$ be a minimum edge cover of the tree $T$. Since $T$ is bipartite, $|F|=\rho(G)=\alpha(T)$. Choose a leaf $r$ of $T$ and consider $T$ as a rooted tree with $r$ as its root, by which the notions of parent and child can be used in $T$. Let an indicated domination game be played in $T$. We will present a strategy of Dominator, which ensures that at most $\alpha(T)$ vertices will be selected during the game.

In the first move, Dominator indicates vertex $r$, which forces Staller to dominate both $r$ and its child $s$. Note that $rs\in F$, and that after the first move of Staller both endvertices of this edge from $F$ become dominated. We will say that an edge $f\in F$ is {\em saturated} if both of its endvertices are dominated. We claim that Dominator can ensure that  after every move of Staller, a new edge from $F$ becomes saturated. In other words, after the $i^{th}$ move of Staller, Dominator can ensure that at least $i$ edges from $F$ are saturated. This is clearly true after the first move of Staller. In the subsequent moves, as long as there are still some undominated vertices left, Dominator considers the vertex of $T$ which is the closest to $r$ among all undominated vertices. Let this vertex be denoted by $u$, and let $uv\in F$ for a neighbor $v$ of $u$. Now, we consider several possibilities that need to be reflected in Dominator's strategy. With this, Dominator will achieve that a new edge of $F$ becomes saturated after every move of Staller, and, in addition, that the set of dominated vertices induces a connected subgraph of $T$ (there is one exception case to this property, in which case Dominator can achieve that after the subsequent move of Staller the set of dominated vertices again induces a connected subgraph of $T$).

If $v$ is the parent of $u$, then it is clear, by how $u$ is defined, that $v$ has been dominated in some of the previous moves. In this case, the strategy of Dominator is to indicate $u$. In this way, $u$ will become dominated after the following move of Staller, and so the edge $uv$ becomes saturated, by which the number of saturated edges from $F$ increases by one. Additionally, the set of dominated vertices still induces a connected subgraph. Next, assume that $u$ is the parent of $v$. It may happen that $v$ is already dominated. (This is the case when the set of dominated vertices does not induce a tree, and how this happens will be explained in the next case.) If this is indeed so, then Dominator indicates $u$, and in the following move of Staller, $u$ will become dominated and $uv\in F$ will thus become saturated. The remaining possibility is that $v$ is not dominated. By the strategy of Dominator, we can argue that all children of $v$ are undominated at this point. Dominator's strategy in this case is to indicate $v$. If $v$ or $u$ is selected by Staller in the next move, then $uv\in F$ becomes saturated and the set of dominated vertices is connected, as desired.
Otherwise, a child $x$ of $v$ is selected by Staller, and let $xy$ be the edge of $F$, which covers $x$. Note that either $y=v$ or $y$ is a child of $x$. In either case, the edge $xy\in F$ becomes saturated. In this case (that is, when a child $x$ of $v$ is selected by Staller), we are in the situation when the set of dominated vertices is not connected, but, as explained earlier,  Dominator's next move is to indicate $u$. In this way, $u$ becomes dominated, and since $v$ was dominated in the preceding move, the edge $uv\in F$ becomes saturated. Thus, after these two consecutive moves, the set of vertices in $D$ that are dominated induces a connected subgraph again.

Since after each move of Staller, the number of saturated edges from $F$ increased, we infer that at most $|F|=\alpha(T)$ vertices are chosen by Staller when the game
ends. Since $\gi(T)\ge \alpha(T)$ by Proposition~\ref{prop:diagram}, the stated equality is proved.
\end{proof}


\begin{corollary}
	\label{cor:paths}
	If $n \geq 1$, then $\gi(P_n) = \left \lceil \frac{n}{2} \right \rceil.$
\end{corollary}

In the next result, we prove for yet another family of graphs, namely the grids, that their indicated domination number equals the independence number.

\begin{theorem}
	If $m, n \geq 1$, then $\gi(P_m \cp P_n) = \alpha(P_m \cp P_n) = \left \lceil \frac{mn}{2} \right \rceil$.
\end{theorem}
\begin{proof}
Let $V(P_m)=\{x_1,\ldots, x_m\}$ and $V(P_n)=\{y_1,\ldots,y_n\}$. We divide the proof into two cases, depending on whether both $m$ and $n$ are odd or not.

First, assume that at least one of $m$ and $n$ is an even integer, and, by symmetry, we may assume with no loss of generality that $m$ is even. Consider the partition of $V(P_m \cp P_n)$ into pairs of adjacent columns. Namely, let $C_i=\{x_{2i-1},x_{2i}\}\times V(P_n)$, for all $i\in [m/2]$. The strategy of Dominator, to achieve his goal that at most $mn/2$ vertices will be selected during the indicated domination game on $P_m \cp P_n$, is divided into two phases. In the first phase, he deals with each of the sets $C_i$, proceeding in their natural order, while in the process some of the vertices of $C_i$, for all $i\in [m/2]$, may be left undominated. In the second phase, he indicates the vertices that were left undominated in the first phase one by one.

Let us first present the strategy of Dominator in $C_1$. Whenever Dominator indicates a vertex of $C_1$, that vertex is in the column $\{x_2\}\times V(P_n)$. Dominator can start by indicating an arbitrary vertex of that column, say $(x_2,y_1)$, and then proceeds by indicating vertices of that column by obeying only one rule, which we explain next.
At a given point in the game we say that the pair $\{(x_{2i-1},y_j),(x_{2i},y_j)\}$, where $i\in [m/2]$ and $j\in [n]$, is {\em empty} if neither vertex in the pair  has been dominated up to that point in the game. The strategy of Dominator while dealing with the set $C_1$ in the first phase is to find an arbitrary empty pair $\{(x_1,y_j),(x_2,y_j)\}$, if it exists, and indicate $(x_2,y_j)$, as long as there exists an empty pair in $C_1$. Thus, in her next move, Staller has to dominate $(x_2,y_j)$. She can do this by choosing either $(x_2,y_j)$ or one of its neighbors. In either case, the pair $\{(x_1,y_j),(x_2,y_j)\}$ is no longer empty after her move. Clearly, if the move of Staller is to choose $(x_1,y_j),(x_2,y_{j-1}), (x_2,y_j)$ or $(x_2,y_{j+1})$, then it may happen that additional pairs of $C_1$ become non-empty. In this case, regardless of which of the vertices $(x_1,y_j),(x_2,y_{j-1}), (x_2,y_j)$ or $(x_2,y_{j+1})$ is chosen by Staller, the chosen vertex is the first vertex of its corresponding pair that was chosen during the game, and since its neighboring pairs have become non-empty (if they were not already non-empty before that move), the other vertex of the pair in which the chosen vertex lies will never be selected by Staller during the first phase of the game. It is also possible, that Staller chooses $(x_3,y_j)$ to dominate $(x_2,y_j)$. In this case, $(x_2,y_j)$ becomes dominated and the pair $\{(x_1,y_j),(x_2,y_j)\}$ becomes non-empty, yet the move of Staller choosing $(x_3,y_j)$ will be considered when dealing with $C_2$.
The first step of the game in the first phase, dealing with $C_1$, ends when all pairs of $C_1$ have become non-empty. Note that by the strategy of Dominator, in each pair $\{(x_1,y_j),(x_2,y_j)\}$ at most one of the vertices has been chosen by Staller, and since all pairs in $C_1$ have become non-empty, at most one of the vertices in each pair remains undominated. Clearly, if a vertex from a pair was chosen, then in the corresponding pair both vertices have been dominated. We summarize these observations as follows: after dealing with $C_1$ in the first phase, $\ell_1$ vertices of $C_1$ have been chosen, and at most $n-\ell_1$ vertices from $C_1$ have been left undominated.

Dominator proceeds by dealing with $C_2$, and the only difference from the initial case when dealing with $C_1$ is that some vertices $(x_3,y_j)$ may have already been chosen by Staller. More generally, when Dominator starts to deal with $C_i$, where $2\le i\le m/2$, some vertices  $(x_{2i-1},y_j)$ may have already been chosen by Staller while dealing with $C_{i-1}$. This also implies that the corresponding pairs $\{(x_{2i-1},y_j),(x_{2i},y_j)\}$, and their eventual neighboring pairs, $\{(x_{2i-1},y_{j-1}),(x_{2i},y_{j-1})\}$ and $\{(x_{2i-1},y_{j+1}),(x_{2i},y_{j+1})\}$, are already non-empty when Dominator starts to deal with $C_i$. The strategy of Dominator is the same as when dealing with $C_1$. Notably, Dominator finds an arbitrary empty pair $\{(x_{2i-1},y_j),(x_{2i},y_j)\}$, if it exists, and indicates $(x_{2i},y_j)$, as long as there exists an empty pair in $C_i$. By using the same arguments as in the previous paragraph, Dominator can ensure that after dealing with $C_i$ in the first phase, $\ell_i$ vertices of $C_i$ have been chosen, and at most $n-\ell_i$ vertices from $C_i$ have been left undominated. The first phase is over after Dominator deals with $C_{m/2}$. By that time, Staller has chosen $\sum_{i=1}^{m/2}{\ell_i}$ vertices, while at most  $$\sum_{i=1}^{m/2}{(n-\ell_i)}=\frac{mn}{2}-\sum_{i=1}^{m/2}{\ell_i}$$ vertices of $P_m\cp P_n$  have been left undominated. Now, the second phase begins, in which Dominator indicates  the remaining undominated vertices one by one, in any order. Hence, Staller will additionally select at most $\frac{mn}{2}-\sum_{i=1}^{m/2}{\ell_i}$ vertices, which together with vertices selected in the first phase contributes to at most $\frac{mn}{2}$ chosen vertices during the entire game. Hence, $\gi(P_m\cp P_n)\le \frac{mn}{2}=\alpha(P_m\cp P_n)$, which combined with Proposition~\ref{prop:diagram} implies the stated result.

Second, consider the case when both $m$ and $n$ are odd. In this case, partition $V(P_m\cp P_n)$ into $\frac{m+1}{2}$ sets, namely $C_0=\{x_1\}\times V(P_n)$ and $C_i=\{x_{2i},x_{2i+1}\}\times V(P_n)$, where $1\le i\le \frac{m-1}{2}$. The proof is similar to the proof of the previous case in which $m$ was even; the only difference is in the first step, where Dominator deals with the first column $C_0$, which we explain next.

Dominator uses the following strategy when dealing with $C_0$, which lasts as long as there is an undominated vertex left in $C_0$. The strategy consists of two phases. In the first phase, Dominator starts by indicating $(x_1,y_1)$, and then, in every further step, he indicates the vertex $(x_1,y_{i+2})$ if $i$ is the largest index such that $(x_1,y_i)$ is dominated and $i\le n-2$. 
By the choices of Staller, a vertex of $C_0$ becomes dominated by itself or one of its neighbors (possibly it is dominated by a neighbor from $C_1$, which will then be considered when dealing with $C_1$). If Staller chose $(x_1,y_{i+3})$ in one of her moves, then after the first phase is over, vertex $(x_1,y_{i+1})$ remains undominated. Now, note that vertices chosen by Staller together with vertices that are not yet dominated after the first phase form an independent set of the path induced by the first column. In the second phase, Dominator selects vertices of $C_0$ that have not been dominated in the first phase (if any) one by one, and in this way all vertices of $C_0$ become dominated. By the previous observation, Staller has chosen at most $\alpha(P_n)=\frac{n+1}{2}$ vertices in $C_0$.

The rest of the strategy of Dominator is exactly the same as in the previous case.  Hence we conclude that Dominator can ensure that at most $\frac{(m-1)n}{2}$ vertices will be selected to dominate the vertices in $C_1\cup \cdots \cup C_{\frac{m-1}{2}}$. Together with the bound on the number of vertices from $C_0$ selected in the first step of the game, we infer that at most $$\frac{n+1}{2}+\frac{(m-1)n}{2}=\frac{mn+1}{2}=\left\lceil \frac{mn}{2}\right \rceil$$ vertices will be selected during the game on $P_m\cp P_n$. We readily infer that $\gi(P_m\cp P_n)=\left\lceil \frac{mn}{2}\right \rceil$.	
\end{proof}

The next family of graphs $G$ for which we suspect that $\gi(G)=\alpha(G)$ are cubic bipartite graphs. We checked by computer that the equality holds for all (cubic) bipartite graphs $G$ with $n(G)\le 10$. Next, we extend this result to $n(G)\le 12$.

\begin{proposition} \label{prop:bipartite1012}
If $G$ is a connected, cubic bipartite graph of order at most $12$, then $\gi(G)=\frac{n(G)}{2}=\alpha(G)$.
\end{proposition}
\begin{proof}
By the observations preceding the proposition, it remains to consider the case when $n(G)=12$.
Let $A=\{a_1,\ldots,a_6\}$ and $B=\{b_1,\ldots,b_6\}$ be the partite sets of $G$. For $j \ge 1$, we let $s_j$ denote the vertex chosen by Staller on her $j^{th}$ move
and let $S_j=\{s_1,\ldots,s_j\}$.  In addition,  $d_i$ will denote the vertex indicated by Dominator in his $i^{th}$ move, for $i \ge 2$.  We provide a strategy for
Dominator that  will ensure at most $6$ vertices are chosen by Staller.  Without loss of generality $s_1=a_1$
and $N(a_1)=\{b_1,b_2,b_3\}$.  Dominator indicates $d_2=b_4$.  If $s_2=b_4$, then $|N[S_2]|=8$, which leaves four undominated vertices. At most four more vertices will be
chosen by Staller.  Otherwise, $s_2 \in N(b_4) \cap (A-\{a_1\})$.  Reindexing if necessary we assume $s_2=a_2$.  If $\{b_5,b_6\} \subseteq N(a_2)$, then $S_2$ dominates
eight vertices and thus at most six vertices will be chosen by Staller.  If $a_2$ is adjacent to only one, say $b_5$, of $b_5$ or $b_6$, then Dominator points to $b_6$.
Now, regardless of which vertex Staller chooses from $\{b_6,a_3,a_4,a_5,a_6\}$ it is easy to see that $|N[S_3]|\ge 9$, which implies that Staller will choose at
most $6$ vertices.  Therefore, we may assume that $N(s_2) \cap \{b_4,b_5,b_6\} = \{b_4\}$, and this implies that $|N(a_2) \cap \{b_1,b_2,b_3\}|=2$.
Note that we now have $(N(b_5) \cup N(b_6)) \subseteq \{a_3,a_4,a_5,a_6\}$.  Dominator then points to $d_3=b_5$. If $s_3=b_5$, then $|N[S_3]|=10$ which implies that
Staller can choose at most five vertices when the game has ended.  Thus, we may assume that Staller chooses $s_3 \in N(b_5) \cap \{a_3,a_4,a_5,a_6\}$.  By reindexing if necessary
we assume that $s_3=a_3$.  Similar to the above argument, if $a_3b_6 \in E(G)$, then Staller can choose at most six vertices when the game has ended.  Hence, we may
assume that $b_6 \notin N(a_3)$.  Dominator now indicates $d_4=b_6$.  Staller must choose $s_4 \in \{b_6,a_4,a_5,a_6\}$, and it follows that $|V(G) - N[S_4]| \le 2$. Therefore,
$\gi(G) \le 6$.  The reverse inequality follows by Proposition~\ref{prop:diagram}.
\end{proof}

We could not extend the above reasoning  to graphs of larger order. Thus it remains  open whether $\gi(G)=\frac{n(G)}{2}$ holds for all cubic bipartite graph $G$.

\section{Graphs with $\gi$ larger than $\Gamma$}
\label{sec:powers}

In this section, we present a class of graphs whose indicated domination numbers exceed their upper irredundance numbers (and therefore also the independence number) by an arbitrarily large amount.

Recall that the {\em $k^{\rm th}$ power} $G^k$ of a graph $G$ has $V(G^k)=V(G)$, and $uv\in E(G^k)$ if and only if $d_G(u,v)\le k$, where $d_G$ is the standard (shortest paths) distance in $G$.

\bigskip

\begin{theorem}
	\label{thm:powers-plus}
	For every $n \geq k \geq 2$, the $k^\mathrm{th}$ powers $P_n^k$ of paths satisfy
	$$\gi(P_n^k) = \Theta\Big(\frac{\log k}{k} \, n\Big)$$
	as $n\to\infty$.
	More explicitly,
	$$(\left\lceil \log (k+1) \right\rceil + 1)
	\left\lfloor \frac{n}{4k} \right\rfloor \leq \gi(P_n^k) \leq \frac{\left\lceil \log (k+1) \right\rceil + 2}{2k+2} \, n + \left\lceil \log (k+1) \right\rceil + 2$$
	where $\log$ means logarithm of base $2$.
\end{theorem}

\begin{proof}
	For the lower bound we may assume without loss of generality that $n$ is a multiple of $4k$.
	We divide $P_n^k$ into $\frac{n}{4k}$ sections of $4k$ consecutive vertices.
	Staller's strategy is to consider each section separately.
	Denote the vertices in one such section $S$ by $v_1, \ldots, v_{4k}$.
	Note that the ``middle'' subsection $S^-$ formed by the vertices $v_{k+1}, \ldots, v_{3k}$ can only be dominated by vertices within $S$.
	Let us also use the notation $S'$ for the ``left'' half $v_1, \ldots, v_{2k}$ of $S$, and write $S''$ for its ``right'' half $v_{2k+1}, \ldots, v_{4k}$.
	
	The first time Dominator indicates a vertex in $S$, with no loss of generality it is one of the vertices from $S'$.
	Staller then selects $v_k$.
	This selection dominates the entire $S^-\cap S'$, but leaves the $k$ vertices of $S^-\cap S''$ undominated.
	More generally, no matter what happened during the game before this move, the set $U$ of undominated vertices inside $S$ is a set $\{v_i,v_{i+1},\dots,v_j\}$ of consecutive vertices (or just a singleton $\{v_i\}$), because only some vertices at the right end of $S''$ may possibly be dominated by a selection from the successor section of $S$.
	Here $j>3k$ may occur, but Staller's strategy will handle this situation as if $j\leq 3k$ held.
	So, attention will be restricted to $U'=\{v_i,v_{i+1},\dots,v_{j'}\}$ where $j'=\min(j,3k)$; hence $2k+1\leq i\leq j'\leq 3k$. The crucial point is:
	\begin{itemize}
	 \item[(*)] If $|U'|\geq 2^s$, for an integer $s\geq 0$, then Staller can achieve that at least $s+1$ steps are needed to make the entire $U'$ dominated during the rest of the game.
	\end{itemize}
	We prove this by induction on $s$, the case of $s=0$ being trivial.
	Once (*) is proved, it follows that the $S^-$ subsection inside each of the $\frac{n}{4k}$ sections $S$ requires at least $\left\lceil \log \,(k+1) \right\rceil+1$ moves under a properly chosen strategy of Staller, hence implying the claimed lower bound.
	
	In a general move of the game when at least one selection has already been made in $S$, assume that the set of undominated vertices inside $S$ is $U\subseteq S''$ (also allowing $U\not\subset S^-$, but $U'\subset S^-$ will hold by definition), and let Dominator indicate vertex $v_\ell$ from $U$.
	\begin{itemize}
	 \item[(a)] If $\ell>3k$ (i.e., $v_\ell\in U\setminus U'$), then Staller imagines as if Dominator has indicated $\ell=3k$ and proceeds as in case (c) below.
	 \item[(b)] If $\ell\leq 3k$ and $\ell-i< j-\ell$, then Staller selects $v_{\ell-k}$.
	 \item[(c)] If $\ell\leq 3k$ and $\ell-i\geq j-\ell$, then Staller selects $v_{\ell+k}$.
	\end{itemize}
	
	In either case, Staller's selection is inside $S$, hence it has no influence on the middle subsection of any section other than $S$.
	Moreover, in case (a) the selected vertex is $v_{4k}$, therefore it dominates $v_\ell$.
	Finally, the number of vertices in $U'$ that become dominated by Staller's selection is at most $\lceil \frac{1}{2}|U'| \rceil$.
	Consequently, if $|U'|\geq 2^s$, then at least $2^{s-1}$ vertices remain undominated inside $S^-$ after this move.
	This implies (*) by induction, and completes the proof of the lower bound.
	
	\medskip
	
	As a preparation to the proof of the upper bound, we observe the following.
	\begin{itemize}
	 \item[(**)] If $U=\{v_i,v_{i+1},\dots,v_j\}$ is a set of consecutive undominated vertices, and $j-i+1\leq 2k+1$, then Dominator can achieve in at most $\lceil \log \, (j-i+2) \rceil$ moves that the entire $U$ becomes dominated.
	\end{itemize}

	The strategy is simple:
	
	\begin{enumerate}
	 \item Indicate $v_\ell$, where $ \ell=\lfloor\frac{1}{2}(i+j) \rfloor$.
	 \item Depending on Staller's selection, update $U$ to its part that remains undominated.
	 \item Return to step 1 as long as $U$ is non-empty.
	\end{enumerate}
	
	To verify that this strategy proves (**), it suffices to observe that the undominated vertices of $U$ remain consecutive after each move---this is because $j-i+1\leq 2k+1$ has been assumed---and that their number gets halved in each move.
	
	\medskip
	
	Now we provide a strategy for Dominator, who will proceed from left to right in $P_n^k$.
	The situation before each move can be described with an alternating sequence of sections that we call ``gaps'' $U_i$, consisting of consecutive undominated vertices, and ``intervals'' $D_i$, consisting of consecutive dominated vertices.
	At the beginning we have no intervals, and just one gap $U_1=\{v_1,\dots,v_n\}$.
	Let Dominator play according to the following rules.
	 \begin{enumerate}
	  \item While $v_n$ is undominated and the last gap has at least $2k+2$ vertices:
	   \begin{enumerate}
	    \item Identify the first vertex $v_g$ of the gap that contains $v_n$.
	    \item Indicate the vertex $v_{\ell}$, where $\ell=g+2k+1$.
	    \item Depending on Staller's selection, update the alternating sequence $U_1$, $D_1$, $U_2$, $D_2$, $\dots$ of gaps and intervals.
	   \end{enumerate}
	  \item In each of the gaps $U_1,U_2,\dots,$ apply the strategy described above for (**).
	 \end{enumerate}

	After the first phase, for the sake of a more transparent computation, we split the gap-interval vertex partition into blocks $B_1=(U_1,D_1)$, $B_2=(U_2,D_2)$, $\dots,$ $B_m=(U_m,D_m)$.
	If $v_n$ remains undominated before the second phase, i.e.\ $v_n\in U_m$, we artificially define $D_m$ as the emptyset.
	Observe that except for the last block~$B_m$, each gap $U_i$ has at least 1 and at most $2k+1$ vertices, and each interval $D_i$ has exactly $2k+1$ vertices.
	In particular, also the first block begins with a gap, as $v_1$ is undominated.

	Consider any $B_i$ with $i<m$.
	If $2^s \leq |U_i|< 2^{s+1}$, (**) guarantees that Dominator can achieve that $U_i$ becomes dominated within $s+1 \leq \lceil \log \,(2k+2) \rceil$ moves (and, according to (*), that is the best he can do).
	Moreover, $D_i$ was dominated in just one step.
	Since $|U_i\cup D_i|\geq 2k+1+2^s\geq 2k + 2$ whenever $i<m$, the average number of moves required to dominate one vertex in $B_i$ is
	$$
	  \frac{s+2}{|U_i|+|D_i|} \leq \frac{\lceil \log \,(2k+2) \rceil + 1}{2k + 2}
	$$
	because $s$ is an integer and $2^s \leq |U_i| \leq 2k+1$.
	This upper bound is valid for all vertices in the whole $B_1 \cup \cdots \cup B_{m-1}$.

	To complete the proof of the theorem, it suffices to note that $U_m$ becomes dominated in at most $\left\lceil \log \,(2k+2) \right\rceil = \left\lceil \log \,(k+1) \right\rceil + 1$ moves and $D_m$ in just one move (or none if it is empty).
\end{proof}

Let us state the particular case $k=2$ of the general lower bound separately, as it already has important consequences.

\begin{corollary}
	\label{cor:2nd-power}
	If $n \geq 2$, then $\gi(P_n^2) \geq 3 \lfloor \frac{n}{8} \rfloor$ 
\end{corollary}
	
	\medskip
	
Since all $P_n^2$ are chordal graphs, we have $\IR(P_{n}^2) = \Gamma(P_n^2) =  \alpha(P_{n}^2) = \left \lceil \frac{n}{3} \right \rceil \leq \frac{n+2}{3}$ by Theorem~\ref{thm:topthreechordal}. 
Together with Corollary~\ref{cor:2nd-power}, noting also that $3 \lfloor \frac{n}{8} \rfloor \geq \frac{3}{8}n-\frac{21}{8}$, we get $\gi(P_n^2) - \IR(P_n^2) \geq  \frac{n-79}{24}$, which can be arbitrarily large when $n$ grows.

\begin{corollary}
	\label{cor:gi-IR-plus}
	There exist graphs $G$ such that $\gi(G) - \IR(G)$ is arbitrarily large.
\end{corollary}

The above corollary also implies that there exist graphs $G$ such that $\gi(G) - \Gamma(G)$ is arbitrarily large.

\section{Concluding remarks}
\label{sec:conclude}


In Section~\ref{sec:upper} we characterized the graphs $G$ with minimum degree $\delta$ and order $n \geq 2 \delta + 2$, that satisfy the extremal value $\gi(G) = n - \delta$. It would be interesting to extend the characterization to graphs of smaller order relative to $\delta$.

\begin{question}
For which graphs $G$ of order $n$ and minimum degree $\delta$ such that $n < 2\delta + 2$ does $\gi(G) = n - \delta$ hold?
\end{question}

In Section~\ref{sec:alpha}, we proved for two important families of bipartite graphs (namely trees and grids) that the indicated domination number equals their independence number. In addition, computer check confirmed this to hold for all (bipartite) graphs of order at most $10$. We could not find any bipartite graph that would not have this property, so we pose this as an open problem:

\begin{question}
	Is it true that for every bipartite graph $G$ we have $\gi(G) = \alpha(G)$?
\end{question}

In the same section we also proved that connected, cubic bipartite graphs $G$ of order at most $12$ satisfy $\gi(G)=\frac{n(G)}{2}=\alpha(G)$. We think this result could be generalized to all connected, cubic bipartite graphs, but we were unable to prove it. Therefore we pose the following question.

\begin{question}
Is it true that $\gi(G)=\frac{n(G)}{2}$ for any cubic bipartite graph $G$?
\end{question}

It would also be interesting to find whether there is an upper bound on the indicated domination number in the class of all connected cubic graphs. In particular, is there a constant $C<1$ such that $\gi(G)\le C\cdot n$ holds for all connected cubic graphs $G$? Note that $C \geq \frac12$ due to the $3$-cube and the Petersen graph, as checked by computer.

\paragraph{Acknowledgements.}

This research was supported in part by the National Research, Development and Innovation Office, NKFIH Grants SNN 129364 and FK 132060, by the Slovenian Research and Innovation Agency (ARIS) under the grants P1-0297, J1-2452, J1-3002, J1-4008, N1-0285, N1-0218 and Z1-50003, and by the European Union (ERC, KARST, 101071836). The authors thank the Faculty of Natural Sciences and Mathematics of the University of Maribor for hosting the Workshop on Games on Graphs in June 2023.


\begin{thebibliography}{99}

\bibitem{bacso-2023}
G.~Bacs\'{o}, B.~Bre\v{s}ar, K.~Kuenzel, D.F.~Rall,
Graphs with equal Grundy domination and independence number,
Discrete Optim.\ 48  (2023) Paper 100777, 15~pp.

\bibitem{bagr-07} T.~Bartnicki, J.~Grytczuk, H.~A.~Kierstead, X.~Zhu, The
map coloring game, Amer.\ Math.\ Monthly 144 (2007) 793--803.

\bibitem{bo-1991} H.L.~Bodlaender, On the complexity of some coloring games, Internat.\ J.\ Found.\ Comput.\ Sci.\ 2 (1991) 133--147.

\bibitem{borowiecki+2019connected}
M.~Borowiecki, A.~Fiedorowicz, E.~Sidorowicz,
Connected domination game,
Appl.\ Anal.\ Discrete Math.\ 13 (2019) 261--289.


\bibitem{borowiecki+2007listcol}
M.~Borowiecki, E.~Sidorowicz, Zs.~Tuza,
Game list colouring of graphs,
Electron.\ J.\ Combin.\ 14 (2007) Paper R26, 11 pp.


\bibitem{bresar-2014}
B.~Bre\v{s}ar, T.~Gologranc, M.~Milani\v{c}, D.F.~Rall, R.~Rizzi,
Dominating sequences in graphs,
Discrete Math.\ 336 (2014) 22--36.

\bibitem{book-2021}
B.~Bre\v{s}ar, M.A.~Henning, S.~Klav\v{z}ar, D.F.~Rall,
{\em Domination Games Played on Graphs},
SpringerBriefs in Mathematics, 2021.

\bibitem{bresar-2010}
B.~Bre{\v{s}}ar, S.~Klav\v{z}ar, D.F.~Rall,
Domination game and an imagination strategy,
SIAM J.\ Discrete Math.\ 24 (2010) 979--991.

\bibitem{bujtas-2022a}
Cs.\ Bujt\'{a}s, V.\ Ir\v{s}i\v{c}, S.~Klav\v{z}ar, 
$1/2$-conjectures on the domination game and claw-free graphs, 
European J.\ Combin.\ 101 (2022) 103467.

\bibitem{bujtas-2022b}
Cs.\ Bujt\'{a}s, V.\ Ir\v{s}i\v{c}, S.~Klav\v{z}ar, 
Predominating a vertex in the connected domination game, 
Graphs Combin.\ 38 (2022) paper 77.

\bibitem{bujtas-2021}
Cs.\ Bujt\'{a}s, V.\ Ir\v{s}i\v{c}, S.~Klav\v{z}ar, K.\ Xu, 
On Rall’s $1/2$-conjecture on the domination game, 
Quaest.\ Math.\ 44 (2021) 1711--1727.

\bibitem{fractional-2019}
Cs.\ Bujt\'{a}s, Zs.~Tuza,
Fractional domination game,
Electron.\ J.\ Combin.\ 26 (2019) Paper 4.3, 17 pp.

\bibitem{co-fa-pa-th-1981}  E.J.~Cockayne, O.~Favaron, C.~Payan, and A.~G.~Thomason,
Contributions to the theory of domination, independence and irredundance in graphs,
Discrete Math.\  33  (1981) 249--258.

\bibitem{duchene-2020}
E.~Duch\^{e}ne, V.~Gledel, A.~Parreau, G.~Renault,
Maker-Breaker domination game,
Discrete Math.\ 343 (2020) Paper 111955, 12 pp.		

\bibitem{ga-81}
M.~Gardner,
Mathematical games,
Scientific American 244 (1981) 18--26.

\bibitem{gledel-2020}
V.~Gledel, M.A.~Henning, S.~Klav\v{z}ar, V.~Ir\v{s}i\v{c},
Maker-Breaker total domination game,
Discrete Appl.\ Math.\ 282 (2020) 96--107.

\bibitem{gr-2012} A.~Grzesik, Indicated coloring of graphs,
Discrete Math.\ 312 (2012) 3467--3472.

\bibitem{haynes-1998}
T.W.~Haynes, S.T.~Hedetniemi, P.~Slater,
Fundamentals of Domination in Graphs,
Marcel Dekker Inc., New York, NY, 1998.

\bibitem{henning-2015}
M.A.~Henning, S.~Klav\v{z}ar, D.F.~Rall,
Total version of the domination game,
Graphs Combin.\ 31 (2015) 1453--1462.

\bibitem{jacobsen-1990}
M.S.~Jacobson, K.~Peters,
Chordal graphs and upper irredundance, upper domination and independence,
Discrete Math.\ 86 (1990) 59--69.

\bibitem{kosmrlj-2014}
G.~Ko\v{s}mrlj,
Realizations of the game domination number,
J.\ Comb.\ Optim.\ 28 (2014) 447--461.

\bibitem{mpw-2018} T.~Mahoney, G.J.~Puleo, D.B.~West, 
Online sum-paintability: the slow-coloring game, 
Discrete Math.\ 341 (2018) 1084--1093.

\bibitem{xu-2018a}
K.\ Xu, X.\ Li, 
On domination game stable graphs and domination game edge-critical graphs, 
Discrete Appl.\ Math., 250 (2018) 47--56.

\bibitem{xu-2018b}
K.\ Xu, X.\ Li,  S.~Klav\v{z}ar, 
On graphs with largest possible game domination number, 
Discrete Math.\ 341 (2018) 1768--1777.

\bibitem{tz-2015} Zs.~Tuza, X.~Zhu, Colouring games, in Topics in Chromatic Graph Theory, Cambridge Univ.\ Press, Cambridge, 2016, 304--326.
	

	
\end{thebibliography}
\end{document}